\documentclass[a4paper,11pt]{article}
\usepackage{amsthm,amsmath,amsfonts}

\newtheorem{lemma}{Lemma}
\newtheorem{theorem}{Theorem}
\newtheorem{proposition}{Proposition}
\newtheorem{remark}{Remark}
\newtheorem*{corollary*}{Corollary}

\newcommand*{\transp}{\mathsf{T}}
\newcommand*{\A}{\mathcal{A}}
\newcommand*{\B}{\mathcal{B}}

\newcommand*{\Zplus}{\mathbb{Z}_{+}}
\newcommand*{\NN}{\mathbb{N}}
\newcommand*{\one}{\mathbf{1}}
\newcommand*{\p}{{[p]}}
\newcommand*{\C}{\mathbb{C}}
\newcommand*{\ortho}{\perp}
\newcommand*{\partition}{\vdash}
\newcommand*{\sgn}{\mathrm{sgn}}

\newcommand*{\psd}{\succeq 0}

\newcommand*{\stensor}[2]{\ensuremath{\mathrm{Sym}^{#2}(#1)}}

\newcommand*{\inner}[2]{\ensuremath{\left< #1,#2\right> }}
\renewcommand*{\hom}[2]{\ensuremath{\mathrm{Hom}_G\left(#1,#2\right)}}
\newcommand*{\End}[1]{\ensuremath{\mathrm{End}_G(#1)}}
\newcommand*{\GL}[1]{\ensuremath{\mathrm{GL}(#1)}}
\newcommand*{\mat}[2]{ \ensuremath{{#2}^{#1\times#1}}}
\newcommand*{\sst}[2]{\ensuremath{\mathcal{T}_{#1,#2}}}

\title{Block diagonalization for algebra's associated with block codes}
\author{Dion Gijswijt\footnote{CWI and University of Leiden. Mailing address: CWI, Science Park 123, 1098 XG Amsterdam, The Netherlands. Email: dion.gijswijt@gmail.com.}}

\begin{document}
\maketitle
\begin{abstract}
For a matrix $\ast$-algebra $\B\subseteq \mat{m}{\C}$, denote by $\A:=\stensor{\B}{n}$ the matrix $\ast$-algebra consisting of the elements in the $n$-fold tensor product $\B^{\otimes n}$ that are invariant under permuting the $n$ factors in the tensor product. Examples of such algebras in coding theory include the Bose-Mesner algebra and Terwilliger algebra of the (non)binary Hamming cube and algebras arising in SDP-hierarchies for codes using moment matrices. 
We give a computationally efficient block diagonalization of $\A$ in terms of a given block diagonalization of $\B$. As a tool we use some basic facts from the representation theory of the symmetric group.
\end{abstract}

\section{Introduction}
A matrix $\ast$-algebra is a set $\A\subseteq \mat{n}{\C}$ of matrices that is closed under addition, scalar multiplication, matrix multiplication and $A\mapsto A^\ast$ (taking the conjugate transpose). It is a classical result that any such algebra can be brought into block diagonal form. That is, there exists an isomorphism   
\begin{equation}
\phi:\A\to \bigoplus_{i=1}^t \mat{p_i}{\C}
\end{equation}
of $\ast$-algebras, meaning that $\phi$ is a linear bijection, $\phi(AB)=\phi(A)\phi(B)$ and $\phi(A^\ast)=\phi(A)^\ast$ for all $A,B\in \A$. 

In this paper we will be concerned with constructing in a computationally efficient way such block diagonalizations $\phi$. The motivation comes from semidefinite programming, where block diagonalization has proven to be a valuable tool since it can be used to reduce the complexity of semidefinite programs having a large group of symmetries.

Semidefinite programming is an extension of linear programming that is both very general and at the same time can be performed efficiently (up to given precision), both in theory and in practice. A (complex) semidefinite program is an optimization problem of the form 
\begin{equation}
\text{maximize } \left< C,X\right> \text{ subject to } X\psd \text{ and }\left< A_i, X\right>=b_i \text{ for $i=1,\ldots,m$,}
\end{equation} 
where $X\in \C^{n\times n}$ is an Hermitian matrix variable, $C, A_1,\ldots, A_m\in \C^{n\times n}$ are given Hermitian matrices and $b_1,\ldots, b_m$ are given real numbers. Here $X\psd$ means that $X$ is positive semidefinite and $\left< C,X\right>:=\text{trace } X^\ast C$ is the trace product. Linear programming can be viewed as the special case where the given matrices are diagonal. In the literature, it is more common to consider semidefinite programs where the matrices are real valued and symmetric rather than Hermitian. However, complex semidefinite programming is easily reduced to the real case by mapping each Hermitian matrix $A$ to $\begin{pmatrix}\mathrm{Re }A&-\mathrm{Im }A\\ \mathrm{Im }A&\mathrm{Re }A\end{pmatrix}$, see \cite{Goemans}.  

In recent years, many results have been obtained using semidefinite programming, where the underlying problem exhibits a large group of symmetries (see \cite{Vallentin} for an overview). These symmetries can be exploited to significantly reduce the computational complexity of the semidefinite program at hand. Indeed, let $\A$ be a matrix $\ast$-algebra containing the given matrices $C,A_1,\ldots,A_m$. For example, $\A$ may be the set of matrices invariant under a group of symmetries of the underlying combinatorial problem, or the matrices $C,A_1,\ldots ,A_m$ may belong to the algebra $\A$ associated with an association scheme (or coherent configuration). Then the matrix-variable $X$ can be restricted to $\A$ without changing the optimum. When an explicit $\ast$-isomorhism $\phi:\A\to \bigoplus_i \C^{p_i\times p_i}$ is known, the semidefinite program can be reduced to a semidefinite program in terms of the smaller matrices from $\bigoplus_i \C^{p_i\times p_i}$ using the fact that a $\ast$-isomorphism preserves positive semidefiniteness. When the algebra $\A$ is commutative, for instance when $\A$ is an association scheme, the optimization problem is hence reduced to a linear program in a number of variables equal to the dimension of $\A$.

When applying block diagonalization to semidefinite programs, it is necessary to have a $\ast$-isomorphism available that can be effectively computed. When $\A$ consists of the matrices invariant under the action of a group, the theory of group representations is the tool for deriving an explicit block diagonalization of the invariant algebra $\A$, see \cite{Vallentin}. 

In \cite{Gijswijt,Laurent,Schrijver} semidefinite programming is used to obtain bounds for error-correcting codes. There the symmetry reduction is essential, reducing the size of the matrices from exponential to polynomial in the size of the input. The underlying combinatorial problem involves strings of length $n$ over a finite set of symbols and the main symmetry comes from permuting the $n$ positions in a string. This leads to the consideration of algebras of a specific form.
 
Given a matrix $\ast$-algebra $\B$ and an integer $n$, the algebra $\B^{\otimes n}$ admits an action of the symmetric group $S_n$ by defining $\sigma (B_1\otimes\cdots\otimes B_n):=B_{\sigma^{-1}(1)}\otimes\cdots\otimes B_{\sigma^{-1}(n)}$ for $B_1,\ldots,B_n\in \B$, $\sigma\in S_n$ and linearly extending this action to the whole of $\B^{\otimes n}$. The set of symmetric tensors $\stensor{\B}{n}:=\{A\in \B^{\otimes n}\mid \sigma A=A \text{ for all $\sigma\in S_n$}\}$ is a matrix $\ast$-subalgebra of $\B^{\otimes n}$. 

In this paper we derive an explicit block diagonalization of $\A:=\stensor{\B}{n}$ in terms of a given block diagonalization for $\B$. The resulting block diagonalization can be computed in polynomial time from the block diagonalization of $\B$. When $\B$ and its block diagonalization are defined over the reals (rationals), also the obtained block diagonalization of $\A$ is defined over the reals (rationals), up to a scaling by square roots of rationals in the latter case. 

In Section \ref{sec:example}, we will consider some examples, including the Terwilliger algebra of the binary- and nonbinary Hamming cube taking respectively $\B=\mat{2}{\C}$ and $\B=\mat{2}{\C}\oplus \mat{1}{\C}$.

\section{Preliminaries}\label{sec:prelim}
\subsection{Notation}
We denote the natural numbers by $\NN:=\{0,1,2,\ldots\}$ and the positive integers by $\Zplus$. For $k\in \Zplus$, we denote $[k]:=\{1,\ldots,k\}$. For finite sets $S,I$ and a \emph{word} $a\in S^I$, the \emph{weight} of $a$, denoted $w(a)\in \NN^S$ is given by $w(a)_s:=|\{i\in I\mid a_i=s\}|$ and counts the number of occurences of each element $s\in S$ in the word $a$.  

Let $V$ be a vector space (say over the complex numbers) with basis $v_1,\ldots,v_k$. A basis of the $n$-fold tensor product $V^{\otimes n}$ is given by $(v_a)_{a\in [k]^n}$, where $v_a=v_{a_1}\otimes\cdots\otimes v_{a_n}$. The subspace $\stensor{V}{n}\subseteq V^{\otimes n}$ consisting of the symmetric tensors has a basis indexed by all decompositions $\mu\in \NN^k_n$ of $n$ into $k$ parts: $v_{\mu}:=\sum_{a\in [k]^n, w(a)=\mu} v_a$. 
  
\subsection{Representation theory}
In this section we recall some basic facts from the representation theory of finite groups. These results can be found in most textbooks on group representations such as the book by Fulton and Harris \cite{Fulton}. Given some finite set $S$ and a finite group $G$ acting on $S$, we derive a $\ast$-algebra isomorphism between the set of complex $S\times S$ matrices that are invariant under the action of $G$, and a direct sum of full matrix algebras. This isomorphism is expressed in terms of representations of $G$. 

Let $V$ be a vector space. All vector spaces that we consider will be finite dimensional over the field of complex numbers. By $\GL{V}$ we denote the set of all invertible linear transformations of $V$ to itself. Let $G$ be a finite group. Then $V$ is called a \emph{$G$-module} if there is a group homomorphism $\rho: G\to \GL{V}$. That is, there is an action of $G$ on $V$ such that $g(c\cdot v+c'\cdot w)=c\cdot g(v)+c'\cdot g(w)$ for all $g\in G$, $c,c'\in \C$ and $v,w\in V$.

A \emph{$G$-homomorphism} from $V$ to $W$, is a linear map $\phi:V\to W$ that respects the action of $G$: $\phi(g(v))=g(\phi(v))$ for all $v\in V$ and $g\in G$. If $\phi$ is a bijection, then $V$ and $W$ are isomorphic (as $G$-modules) and we write $V\cong W$. The set of homomorphisms from $V$ to $W$ is denoted by $\hom{V}{W}$ and we denote $\End{V}:=\hom{V}{V}$.
 
Let $\inner{\cdot}{\cdot}$ be a $G$-invariant inner product on $V$, that is $\inner{x}{y}=\inner{g(x)}{g(y)}$ for all $g\in G$ and $x,y\in V$. Such an inner product exists: take $\inner{x}{y}_G:=\sum_{g\in G}\inner{g(x)}{g(y)}$ for any inner product $\inner{\cdot}{\cdot}$. With respect to this inner product, the algebra $\End{V}$ becomes a $\ast$-algebra, where $A^{\ast}$ is the adjoint of $A\in \End{V}$. 

A module $V$ is called \emph{irreducible} if $V$ has exactly two submodules: $\{0\}$ and $V$ itself. If $W\subseteq V$ is a submodule, then also $W^\ortho:=\{x\in V\mid \inner{x}{y}=0 \text{ for all $y\in W$}\}$ is a submodule. Hence it follows that $V$ can be decomposed into pairwise orthogonal irreducible submodules.
\begin{theorem}{(Maschke's Theorem)}
Let $V$ be a $G$-module. Then $V$ has an orthogonal decomposition
\begin{equation}\label{maschke}
V=\bigoplus_{\lambda=1}^k V_\lambda,\quad V_\lambda=\bigoplus_{i=1}^{m_\lambda} V_\lambda^i,
\end{equation}
where $V_\lambda^i\cong V_\lambda^j$ for all $i,j=1,\ldots,m_\lambda$ and the $V_\lambda^1$ ($\lambda=1,\ldots,k$) are pairwise nonisomorphic irreducible submodules of $V$.
\end{theorem}

The set of homomorphisms is controlled by Schur's Lemma.
\begin{theorem}{(Schur's Lemma)}
Let $V$ and $W$ be irreducible $G$-modules. Then either there is an isomorphism $\phi:V\to W$ and $\hom{V}{W}$ consists of the scalar multiples of $\phi$, or $\hom{V}{W}$ consists of the zero map only.
\end{theorem}

Together, Schur's Lemma and Maschke's Theorem imply a block diagonalization $\End{V}\cong \bigoplus_{\lambda=1}^k \mat{m_\lambda}{\C}$. To make this explicit, let $U_1,\ldots, U_k$ be a complete set of irreducible submodules of $V$, say $U_\lambda\cong V_\lambda^1$ for $\lambda=1,\ldots, k$. Let $e_\lambda\in U_\lambda\setminus\{0\}$, $W_\lambda:=\{Ae_\lambda\mid A\in \End{V}\}$, and let $B_\lambda$ be an orthonormal base of $W_\lambda$. Then we obtain a block diagonalization of $\End{V}$ as follows.
\begin{theorem}\label{endV}
The map
\begin{eqnarray}
\psi:\End{V}&\to &\bigoplus_{\lambda=1}^k \mat{m_\lambda}{\C},\\
A&\mapsto& \bigoplus_{\lambda=1}^k \left(\inner{Ab}{b'}\right)_{b,b'\in B_\lambda}\nonumber
\end{eqnarray}
is a $\ast$-algebra isomorphism.
\end{theorem} 
\begin{proof}
For each $\lambda$, the map $A\mapsto \left(\inner{Ab}{b'}\right)_{b,b'\in B_\lambda}$ is a $\ast$-algebra homomorphism since it maps $A$ to its restriction on $W_\lambda$ written as a matrix with respect to an orthonormal base of $W_\lambda$. Therefore $\psi$ itself is a $\ast$-algebra homomorphism. 

To show injectivity of $\psi$, suppose that $\psi(A)=0$. Consider an arbitrary component $V_\lambda^i$ in Maschke's decomposition (\ref{maschke}) and let $B:U_\lambda\to V_\lambda^i$ be a $G$-isomorphism. Then $B$ can be viewed as an element of $\End{V}$ by setting $Bx:=0$ for all $x\in U_\lambda^\ortho$. Since $\psi(A)=0$ it follows that $ABe_\lambda=0$ and hence 
\begin{equation}
A\cdot V_\lambda^i=AB\cdot U_\lambda=AB\cdot \C Ge_\lambda=\C G\cdot ABe_\lambda=\{0\}.
\end{equation}
Since this holds for all $\lambda$ and all $i$, $A$ is the zero map.
Surjectivity follows since 
\begin{eqnarray}
\dim \End{V} &=&\dim \hom{\bigoplus_\lambda \bigoplus_{i=1}^{m_\lambda}V_\lambda^i}{\bigoplus_\mu \bigoplus_{j=1}^{m_\mu}V_\mu^j}\\
&=&\sum_{\lambda,\mu}\sum_{i=1}^{m_\lambda}\sum_{j=1}^{m_\mu}\dim\hom{V_\lambda^i}{V_\mu^j}\nonumber\\
&=&\sum_\lambda m_\lambda^2,\nonumber
\end{eqnarray}
shows that both the dimension of $\End{V}$ and the dimension of $\bigoplus_\lambda \mat{m_\lambda}{\C}$ equal $\sum_\lambda m_\lambda^2$. Hence $\psi$ is indeed a bijection.
\end{proof}

The theorem shows that once a nonzero element $e_\lambda$ in each irreducible submodule (up to equivalence) is identified, the block diagonalization of $\End{V}$ can be computed explicitly once we can evaluate the inner products $\left<AA'e_\lambda,A''e_\lambda\right>$ for any $e_\lambda$ and any $A,A',A''$ from a suitable basis of $\End{V}$.

\section{The algebra $\stensor{\mat{p}{\C}}{n}$}
Consider the set $\p^n$ of words of length $n$ with symbols in the alphabet $\p=\{1,\ldots,p\}$. For example, taking $p=2$ gives the set of all binary words of length $n$ written with symbols $1$ and $2$. The symmetric group $S_n$ acts on $\p^n$ by $(\sigma a)_i:=a_{\sigma^{-1}(i)}$ for all $a\in \p^n$, $\sigma\in S_n$ and $i=1,\ldots,n$. This induces a linear action of $S_n$ on $\C^{\p^n}$ by defining $(\sigma x)_a:=x_{\sigma^{-1}a}$ for $x\in \C^{\p^n}$. Similarly, there is an induced linear action of $S_n$ on the set of $\p^n\times\p^n$ matrices $M$, by defining $(\sigma M)_{a,b}:=M_{\sigma^{-1} a,\sigma^{-1} b}$.

In this section, we give a block diagonalization of the the matrix $\ast$-algebra
\begin{equation}
\A=\{A\in \C^{\p^n\times \p^n}\mid \sigma A=A \text { for all $\sigma\in S_n$}\}
\end{equation}
of all $S_n$-invariant matrices in $\C^{\p^n\times \p^n}$. That $\A$ is indeed a matrix $\ast$ algebra follows from the facts that $S_n$ acts linearly on $\C^{\p^n\times\p^n}$, $\sigma A^\ast=(\sigma A)^\ast$ and $\sigma (AB)=(\sigma A)(\sigma B)$ for all $\sigma\in S_n$ and $A,B\in \C^{\p^n\times\p^n}$.

For any two words $a,b\in \p^n$, define $D(a,b)\in \NN^{p\times p}$ by
\begin{equation}
(D(a,b))_{i,j}:=|\{k\mid a_k=i, b_k=j\}|.
\end{equation}
Clearly, for $a,b,a',b'\in \p^n$ we have $D(a,b)=D(a',b')$ if and only if $a'=\pi(a), b'=\pi(b)$ for some $\pi\in S_n$. Let 
\begin{equation}
P(n,p):=\{ D\in \NN^{p\times p}\mid \one^{\transp}D\one=n\}.
\end{equation}
For $D\in P(n,p)$, let $A_D\in \A$ be given by
\begin{equation}
(A_D)_{a,b}:=\begin{cases}1& \text{if\ }D(a,b)=D,\\0&\text{otherwise}.\end{cases}
\end{equation}
Observe that $A_D^\transp=A_{D^\transp}$.
\begin{proposition}
The matrices $A_D$ with $D\in P(n,p)$, form a basis for $\A$ (as a complex linear space) and the dimension of $\A$ equals $\tbinom{n+p^2-1}{p^2-1}$.
\end{proposition}
\begin{proof}
The matrices $A_D$ are nonzero matrices with disjoint support and hence linearly independent. For any matrix $A\in\A$, the value of an entry $A_{a,b}$ only depends on $D(a,b)$ and is therefore a linear combination of the matrices $A_D$.  
\end{proof}

Since $\A$ is closed under multiplication, there exist numbers $c_{L,M}^N, L,M,N\in P(n,p)$ such that $A_LA_M=\sum_N c_{L,M}^N A_N$. Although we will not need the numbers $c_{L,M}^N$, we mention the following fact for completeness.
\begin{proposition}
The numbers $c_{L,M}^N$ are given by 
\begin{equation}
c_{L,M}^N=\sum_B \prod_{r,t=1}^p \tbinom{N_{r,t}}{B_{r,1,t},\ldots,B_{r,p,t}},
\end{equation}
where the sum runs over all $B\in \NN^{p\times p\times p}$ that satisfy $\sum_k B_{r,s,k}=L_{r,s}$, $\sum_k B_{k,s,t}=M_{s,t}$ and $\sum_k B_{r,k,t}=N_{r,t}$ for all $r,s,t\in \p$.
\end{proposition}
\begin{proof}
For given words $a,c\in [p]^n$ with $D(a,c)=N$, the number $c_{L,M}^N$ equals the number of words $b\in [p]^n$ with $D(a,b)=L$ and $D(b,c)=M$. Fixing the number $B_{r,s,t}$ of positions $i=1,\ldots,n$ for which $a_i=r, b_i=s, c_i=t$, the number of feasible $b$ equals $0$ unless $\sum_k B_{r,s,k}=L_{r,s}$, $\sum_k B_{r,k,t}=N_{r,t}$ and $\sum_k B_{k,s,t}=M_{s,t}$ for all $r,s,t$. In the latter case the solutions $b$ are obtained by partitioning each set $S_{r,t}:=\{i\mid a_i=r,c_i=t\}$ into subsets of size $B_{r,s,t}$ for $s=1,\ldots p$ to choose the entries of $b$ on $S_{r,t}$. 
\end{proof}

\subsection{Block diagonalization of $\stensor{\mat{p}{\C}}{n}$}\label{subsec:blockdiag}
The action of $S_n$ on $\C^{\p^n}$ gives $\C^{\p^n}$ the structure of an $S_n$-module. Identifying matrices in $\C^{\p^n\times\p^n}$ with the corresponding linear maps in $\GL{\C^{\p^n}}$, the subalgebra $\A$ is identified with $\End{\C^{\p^n}}$. Using the representations of the symmetric group, we will find an explicit block diagonalization of $\A$.

A partition $\lambda$ of $n$, written $\lambda\partition n$, is a sequence of nonnegative integers $\lambda_1\geq\lambda_2\geq\cdots\geq \lambda_n$ with $n=\lambda_1+\lambda_2+\cdots+\lambda_n$. Partition $\lambda$ is said to have $k$ \emph{parts} if exactly $k$ of the numbers $\lambda_i$ are nonzero. A \emph{Ferrers diagram} of \emph{shape} $\lambda$ is an array of $n$ boxes, taking the first $\lambda_i$ boxes from the $i$-th row of an $n\times n$ matrix of boxes. The $j$-th box in row $i$ is referred to as the box in position $(i,j)$. We also number the boxes from $1$ to $n$ according to the lexicographic order on their positions: the box in position $(i,j)$ has number $\lambda_1+\cdots+\lambda_{i-1}+j$. The \emph{dual partition} $\lambda^*$ of $n$ gives the column lengths of a Ferrers diagram of shape $\lambda$: $\lambda^*_i=|\{j\mid \lambda_j\geq i\}|$.

A \emph{tableau} $t$ of shape $\lambda$ is a filling of the boxes in the Ferrers diagram with integers. Here all entries will be from the set $\p$. The number in box $k$ is denoted by $t(k)$. Fixing a partition $\lambda\partition n$, each element $a$ of $\p^n$ is identified with the tableau $t$ given by $t(k):=a_k$. For an element $\sigma\in S_n$ and a tableau $t$, we define $(\sigma t)(k):=t(\sigma^{-1}(k))$. A tableau $t$ is called \emph{semistandard} if the entries are increasing in each column and non-decreasing in each row. We denote the set of semistandard tableaux of shape $\lambda$ with entries from $[p]$ by $\sst{\lambda}{p}$. 

Given a partition $\lambda\partition n$, we can associate with $\lambda$ two subgroups of $S_n$. The group $C_\lambda$ consists of the permutations of boxes within the columns of the diagram, and the group $R_\lambda$ consists of the permutations that permute the boxes within the rows. Two $\lambda$-tableaux $t,t'$ are called \emph{(row-)equivalent}, written $t\sim t'$, when $t'=\pi t$ for some $\pi\in R_\lambda$.

Recall that the complex vector space $V=\C^{\p^n}$ is an $S_n$-module. We denote by $\chi^a\in V$ the standard basis vector corresponding to the word $a$. Then the action of $S_n$ on $V$ is given by 
\begin{equation}
\sigma (\sum_{a\in \p^n}c_a\chi^a):=\sum_{a\in \p^n}c_a\chi^{\sigma a}
\end{equation}
for all $\sigma\in S_n$ and $c_a\in \C$ for all $a\in \p^n$.

Let $\lambda$ be a partition of $n$ into at most $p$ parts. Define the tableau $t_\lambda$ by filling the positions in the $i$-th row with $i$'s. Given any tableau $t$ of shape $\lambda$, we define $e_t\in V$ by
\begin{equation}
e_t:=\sum_{\sigma\in C_\lambda} \sgn(\sigma)\sum_{t'\sim t} \chi^{\sigma(t')}. 
\end{equation}
Define $S^\lambda:=\C S_n\cdot e_\lambda$, where $e_\lambda:=e_{t_\lambda}=\sum_{\sigma\in C_\lambda}\sgn(\sigma) \chi^{\sigma(t_\lambda)}$.

\begin{theorem}\label{spechtmodules}
The $S^\lambda$, where $\lambda$ runs over all partitions of $n$ into at most $p$ parts, form a complete set of pairwise non-isomorphic, irreducible submodules of $V$.
\end{theorem}
The irreducible modules $S^\lambda$ are called \emph{Specht modules}. The proof is standard and can be found for example in \cite{Sagan}.

Let $D\in P(n,p)$ and let $\lambda=D^\transp\one$. If $\lambda$ is nonincreasing, we view it as a partition of $n$ and make a tableau $t(D)$ of shape $\lambda$ and weight $\mu=D\one$ as follows: the $i$-th row of $t(D)$ contains $D_{j,i}$ symbols $j$, and we make the rows of $t(D)$ non-decreasing. Observe that $D(t(D),t_{\lambda})=D$. We have the following lemma.
\begin{lemma}
Let $A=A_D\in \A$ and let $\lambda':=D^{\transp}\one$. Then
\begin{equation}
Ae_\lambda=\begin{cases}
            e_{t(D)} & \text{if } \lambda'=\lambda\\
            0 & \text{otherwise.}
           \end{cases}
\end{equation}
\end{lemma}
\begin{proof}
Recall that we identify $[p]^n$ with the set of all tableau of shape $\lambda$ and entries from $[p]$. For any such tableau $t$, we have $A\chi^t=\sum_{s\mid D(s,t)=D}\chi^s$. Since $D(s,t)^\transp\one=w(t)$, this sum equals zero when $w(t)\neq \lambda'$. In particular, 
\begin{equation}
A\chi^{t_\lambda}=\begin{cases}0&\text{if $\lambda\neq \lambda'$,}\\ \sum_{s\sim t(D)}\chi^s&\text{if $\lambda=\lambda'$}.\end{cases}
\end{equation}
The lemma now follows by writing out the definition of $Ae_\lambda$:
\begin{eqnarray}
Ae_\lambda&=&\sum_{\sigma\in C_\lambda} \sgn(\sigma) A\chi^{\sigma(t_\lambda)}\\
&=&\sum_{\sigma\in C_\lambda} \sgn(\sigma) \sigma(A\chi^{t_\lambda}).\nonumber
\end{eqnarray}
\end{proof}
The lemma implies that the linear space $\A e_\lambda$ is spanned by the vectors $e_t$ (where $t\in [p]^n$ is a tableau of shape $\lambda$). The following theorem selects a subset of the tableaux to obtain a basis.
\begin{theorem}
The $e_t$, $t\in\sst{\lambda}{p}$ constitute a basis of $\mathrm{span}\{e_t\mid t\in [p]^n\text{ is a tableau of shape $\lambda$\}}.$
\end{theorem}
The proof is standard and can be found for example in \cite{Sagan}.

In general the basis $\{e_t\}$ ($t\in\sst{\lambda}{p}$) of $\A e_\lambda$, is not orthonormal. However, let $B:=(e_t)_t$ be the matrix with the $e_t$ as columns and let $G_\lambda:=B^\transp B$ be the Gram matrix of the $e_t$. Take a Cholesky decomposition $(G_\lambda)^{-1}=:R_\lambda R_\lambda^{\transp}$, then the columns of $BR_{\lambda}$ form an orthonormal base of $\A e_\lambda$. 

\begin{theorem}
The map
\begin{eqnarray}
\psi:\A&\to& \bigoplus_{\lambda}\mat{m_\lambda}{\C}\\
A&\mapsto&\bigoplus_{\lambda}R_\lambda^{\transp}\left(\inner{Ae_s}{e_{t}}\right)_{s,t} R_\lambda\nonumber
\end{eqnarray}
is a $\ast$-isomorphism.
\end{theorem}
\begin{proof}
This follows directly from Theorem \ref{endV} since the $e_t$, $t\in \sst{\lambda}{p}$ form a basis of the space $\A e_\lambda$, where the $e_\lambda$ generate the irreducible submodules of $V$.
\end{proof}

\begin{remark}
Although the vectors $e_t$ have length exponential in $n$, we can compute their inner products efficiently (see next section). This implies that we can find the Gram matrices $G_\lambda$ efficiently. However, in applications to semidefinite programming, we only need $\psi$ to preserve positive semidefiniteness. Hence we can neglect the matrices $R_\lambda$ and obtain a bijection
\begin{eqnarray}
\psi':\A&\to& \bigoplus_{\lambda}\mat{m_\lambda}{\C}\\
A&\mapsto&\bigoplus_{\lambda}\left(\inner{Ae_s}{e_{t}}\right)_{s,t}\nonumber
\end{eqnarray}
that preserves positive semidefiniteness.
\end{remark}

\subsection{Computing $\psi$}
In this section, we wil show how the map $\psi$ can be computed efficiently. That is, given a partition $\lambda$ of $n$ into at most $p$ parts, and given semistandard tableaux $s,t\in\sst{\lambda}{p}$ of shape $\lambda$, we compute for every $D\in P(n,p)$, the inner product $\inner{A_De_s}{e_t}$.

For $A\in\C^{[p]^2\times [p]^n}$ the linear map $\A\to \C$ given by $A_D\mapsto \inner{A_D}{A}$ is conveniently expressed using polynomials in the $p^2$ variables $x_{i,j},\ i,j=1,\ldots,p$: 
\begin{equation}
f_A:=\sum_{D\in P(n,p)} x^D \inner{A_D}{A},
\end{equation}
where the shorthand notation $x^D:=\prod_{i,j=1}^p x_{i,j}^{D_{i,j}}$ is used.

Define for $k=0,1,\ldots,p$ the polynomials $Q_k$ by setting $Q_0(x):=1$ and  
\begin{equation}
Q_k(x):=k! \det \begin{pmatrix}x_{1,1}&\cdots&x_{1,k}\\ \vdots&\ddots&\vdots\\x_{k,1}&\cdots&x_{k,k}\end{pmatrix}
\end{equation}
for $k=1,2,\ldots,p$. The polynomials $Q_k$ have at most $p!$ terms and can, for fixed $p$, be computed in constant time. Given a partition $\lambda\partition n$ into at most $p$ parts, define the polynomial $P_\lambda$ by 
\begin{equation}
P_\lambda:=\prod_{i=1}^n Q_{\lambda^*_i}=Q_1^{\lambda_1-\lambda_2}\cdots Q_{p-1}^{\lambda_{p-1}-\lambda_p}\cdot Q_p^{\lambda_p}.
\end{equation}
For fixed $p$, computing $P_{\lambda}$ can be done in time $O(n\dim(\A))$.
\begin{proposition}
Let $\lambda\partition n$ be a partition with at most $p$ parts. Then 
\begin{equation}
P_{\lambda}=f_{e_\lambda e_\lambda^\transp}.
\end{equation}
\end{proposition}
\begin{proof}
Writing out the definition of $A_D$ and $e_\lambda$ we obtain
\begin{eqnarray}
\sum_{D\in P(n,p)} x^D\cdot \inner{A_D}{e_\lambda e_\lambda^\transp}&=&\sum_{D\in P(n,p)} x^D \sum_{\rho,\tau\in C_\lambda} \sgn(\rho)\sgn(\tau) (A_D)_{\rho t_\lambda,\tau t_\lambda}\\
&=&\sum_{D\in P(n,p)} x^D\cdot |C_\lambda|\sum_{\sigma\in C_\lambda} \sgn(\sigma)(A_D)_{\sigma t_\lambda,t_\lambda}\nonumber\\
&=&|C_\lambda|\sum_{\sigma\in C_\lambda} \sgn(\sigma)x^{D(\sigma t_\lambda,t_\lambda)}.\nonumber
\end{eqnarray}
Here we used the substitution $\sigma:=\tau^{-1}\rho$ so that 
\begin{equation}
\sgn(\rho)\sgn(\tau) (A_D)_{\rho t_\lambda,\tau t_\lambda}=\sgn(\sigma) (A_D)_{\sigma t_\lambda,t_\lambda}.
\end{equation} 

Any $\sigma\in C_\lambda$ corresponds to an $n$-tuple $(\sigma_1,\ldots,\sigma_n)$, where $\sigma_i\in S_{\lambda_i^*}$ is the restriction of $\sigma$ to the $i$-th column. Observe that $(t_\lambda)_{i,j}=i$ and $(\sigma t_\lambda)_{i,j}=\sigma_j^{-1}(i)$ and hence 
\begin{equation}
x^{D(\sigma t_\lambda,t_\lambda)}=\prod_{j=1}^n\prod_{i=1}^{\lambda^*_j}x_{\sigma_j^{-1}(i),i}.
\end{equation}
Using this, we obtain:
\begin{eqnarray}
|C_\lambda|\sum_{\sigma\in C_\lambda} \sgn(\sigma)x^{D(\sigma t_\lambda,t_\lambda)}&=&|C_\lambda|\prod_{j=1}^n (\sum_{\sigma_j\in S_{\lambda^*_j}}\sgn(\sigma_j)\prod_{i=1}^{\lambda^*_j}x_{\sigma_j^{-1}(i),i})\\
&=&|C_\lambda|\prod_{j=1}^n \frac{1}{\lambda^*_j !}Q_{\lambda^*_j}\nonumber\\
&=&P_\lambda.\nonumber
\end{eqnarray}
\end{proof}


Since we can perform multiplication in $\A$ in polynomial time, we now have a polynomial time algorithm to compute $\inner{Ae_t}{e_{t'}}=\inner{A_{D(t',t_\lambda)}^\transp A(D) A_{D(t,t_\lambda)}}{e_\lambda e_\lambda^\transp}$ for any $t,t'\in \sst{\lambda}{p}$ and $D\in P(n,p)$, hence obtaining the map $\psi$ explicitly.

Below we show how we can speed up the computation, by avoiding the computationally costly multiplication in $\A$. 

For $i,j\in \p$, let $A_{i\to j}\in\A$ be defined by 
\begin{equation}
(A_{i\to j})_{a,b}:=
\begin{cases}
1 & \text{if there is a $k$ such that $a_k=i,b_k=j$ and $a_l=b_l$ for all $l\neq k$,}\\
0 & \text{otherwise}.
\end{cases}
\end{equation}

For $i,j\in [p]$, let $E_{i,j}$ be the $[p]\times [p]$ matrix with $(E_{i,j})_{i,j}=1$ and all other entries equal to $0$. 
\begin{proposition}\label{Aijaction}
For any $D\in P(n,p)$ we have:
\begin{eqnarray}
A_DA_{i\to j}&=&\sum_{k\mid A_{k,i}>0} (D_{k,j}+1)A_{D-E_{k,i}+E_{k,j}},\\
A_{i\to j}A_D&=&\sum_{k\mid A_{j,k}>0} (D_{i,k}+1)A_{D-E_{j,k}+E_{i,k}}.\nonumber
\end{eqnarray}
\end{proposition}
\begin{proof}
To prove he first line, consider two words $a,b\in\p^n$. The entry $(A_DA_{i\to j})_{a,b}$ equals the number of $c\in\p^n$ such that $D(a,c)=D$ and $b$ is obtained from $c$ by changing an $i$ into a $j$ in some position $h$. If $a_h=k$, then $D(a,b)=D-E_{k,i}+E_{k,j}$. 
When $D(a,b)=D-E_{k,i}+E_{k,j}$, the number of possible $c$ equals $D(a,b)_{k,j}=D_{k,j}+1$.

For the second line, observe that $A_{i\to j}A_D=(A_{D^\transp}A_{j\to i})^\transp$.
\end{proof}

Given a decomposition $\mu=(\mu_1,\ldots,\mu_p)$ of $n$ ($\mu_i\geq 0$), let $I_{\mu}:=A_{\mu_1E_{1,1}+\cdots+\mu_pE_{p,p}}$. 
\begin{proposition}\label{Decompose}
If $D\in P(n,p)$ is lower triangular with $\mu:=D^\transp \one$, then
\begin{equation} \left(\prod_{i>j}D_{i,j}!\right)A_D=(\prod_{j=1}^{p-1}\prod_{i=j+1}^p A_{i\to j}^{D_{i,j}}) I_{\mu}.
\end{equation}

\end{proposition}
\begin{proof}
For given $i\neq j$ and $D\in P(n,p)$ with $D_{j,j}\geq 1$ and $D_{j,k}=0$ for all $k\neq j$, Proposition (\ref{Aijaction}) gives 
\begin{equation}\label{onestep}
A_{i\to j}A_D=(D_{i,j}+1)A_{D-E_{j,j}+E_{i,j}}.
\end{equation}
Let $(i_1,j_1),\ldots, (i_m,j_m)$ be $m$ pairs of indices with $i_k\neq j_k$ for all $k$ and $i_k>j_l$ whenever $k<l$, and let $s_1,\ldots,s_m$ be nonnegative integers. Let $D:=\mu_1E_{1,1}+\cdots+\mu_pE_{p,p}+\sum_{k=1}^m s_i(E_{i_k,j_k}-E_{j_k,j_k})$ and suppose that $D$ is nonnegative. Then from (\ref{onestep}) we obtain  
\begin{equation}
A_{i_m\to j_m}^{s_m}\cdots A_{i_1\to j_1}^{s_1}I_{\mu}=s_m!\cdots s_1! A_D,
\end{equation}
by induction on $m$. This implies the statement in the proposition. 
\end{proof}

Let us denote 
\begin{equation}
d_{i\to j}:=\sum_{s=1}^p x_{i,s}\frac{\partial}{\partial x_{j,s}}, \quad d^*_{i\to j}:=\sum_{s=1}^p x_{s,j}\frac{\partial}{\partial x_{s,i}}.
\end{equation}
In terms of polynomials, Proposition (\ref{Aijaction}) now gives

\begin{proposition} \label{Aijonpoly}
For $i\not=j$ and a matrix $A$, we have
\begin{eqnarray}
f_{A_{i\to j}A}&=&d_{i\to j} f_A\\
f_{AA_{i\to j}}&=&d^*_{i\to j} f_A.\nonumber
\end{eqnarray}
\end{proposition}
\begin{proof}
We have:
\begin{eqnarray}
f_{AA_{i\to j}}=\sum_D x^D\inner{A_DA_{j\to i}}{A}&=&\sum_D x^D\sum_{k\mid D_{k,i}>0} (D_{k,i}+1)\inner{A_{D-E_{k,j}+E_{k,i}}}{A}\\
&=&\sum_{k\mid D_{k,i}>0} \sum_D x^D(D_{k,i}+1)\inner{A_{D-E_{k,j}+E_{k,i}}}{A}\nonumber\\
&=&\sum_{k\mid D_{k,i}>0} \sum_{D'} x^{D'+E_{k,j}-E_{k,i}} D'_{k,i}\inner{A_{D'}}{A}\nonumber\\
&=&\sum_{k=1}^p \frac{x_{k,j}\partial}{\partial x_{k,i}} f_A.\nonumber
\end{eqnarray}
Here we used the substitution $D':=D-E_{k,i}+E_{k,j}$.
The proof for $f_{A_{i\to j}A}$ is similar.
\end{proof}

\begin{theorem}
Let $t',t''$ be semistandard $\lambda$-tableau and let $D':=D(t',t_\lambda), D'':=D(t'',t_\lambda)$. Then
\begin{equation}
\sum_{D\in P(n,p)} x^D\inner{A_De_{t'}}{e_{t''}}=q\cdot \prod_{i>j} (D'_{i,j}!D''_{i,j}!)^{-1},
\end{equation}
where
\begin{equation}
q:=\prod_{j=1}^{p-1}\prod_{i=j+1}^{p} ((d_{i\to j})^{D''_{i,j}}(d^*_{j\to i})^{D'_{i,j}})\circ P_\lambda.
\end{equation}
\end{theorem}
\begin{proof}
We want to compute 
\begin{equation}
\sum_D x^D\inner{A_De_{t'}}{e_{t''}}=\sum_D x^D\inner{A_D}{A_{D''}e_\lambda e_\lambda^\transp A_{D'}^\transp}=f_{A_{D''}E_\lambda A_{D'}^\transp}.
\end{equation}
By Proposition (\ref{Decompose}) we obtain
\begin{eqnarray}
\prod_{i>j} (D'_{i,j}!D''_{i,j}!)\cdot A_{D''}E_\lambda A_{D'}^\transp&=&\prod_{j=1}^{p-1}\prod_{i=j+1}^p A_{i\to j}^{D''_{i,j}}I_\lambda E_\lambda I_\lambda (\prod_{j=1}^{p-1}\prod_{i=j+1}^{p} A_{i\to j}^{D'_{i,j}})^\transp\\
&=&\prod_{j=1}^{p-1}\prod_{i=j+1}^p A_{i\to j}^{D''_{i,j}} E_\lambda (\prod_{j=1}^{p-1}\prod_{i=j+1}^{p} A_{i\to j}^{D'_{i,j}})^\transp\nonumber
\end{eqnarray}
Using Proposition (\ref{Aijonpoly}) we derive
\begin{eqnarray}
\prod_{i>j} (D'_{i,j}!D''_{i,j}!)\cdot f_{A_{D''}E_\lambda A_{D'}^\transp}&=&\prod_{j=1}^{p-1}\prod_{i=j+1}^{p} ((d_{i\to j})^{D''_{i,j}}(d^*_{j\to i})^{D'_{i,j\ast}})\circ P_\lambda.
\end{eqnarray}
\end{proof}

\section{The general case}
Let $\B$ be a complex matrix algebra with basis $R_1,\ldots, R_s$, and let $\A:=\stensor{\B}{n}$. Recall that we have a basis $\{R_\mu\mid \mu\in \NN^s_n\}$ given by $R_\mu:=\sum_{x\in [s]^n\mid w(x)=\mu} \otimes_{i=1}^n R_{x_i}$.

Assume that a block diagonalization
\begin{equation}
\phi:\B\to \bigoplus_{i=1}^t \mat{p_i}{\C},\qquad B\mapsto \bigoplus_{i=1}^t \phi_i(B)
\end{equation}
is given. In the following, we will describe a block diagonalization of $\A$ in terms of $\phi$ and with respect to the basis $\{R_\mu\}$.

First observe that the isomorphism $\phi^{\otimes n}:\B^{\otimes n}\to (\oplus_{i=1}^t \mat{p_i}{\C})^{\otimes n}$ gives an isomorphism $\A\to \stensor{\oplus_{i=1}^t \mat{p_i}{\C}}{n}$ by restriction to $\A$. Now observe, that by removing multiple copies of identical blocks in $\stensor{\oplus_{i=1}^t\mat{p_i}{\C}}{n}$, we have an isomorphism 
\begin{equation}
\stensor{\oplus_{i=1}^t\mat{p_i}{\C}}{n}\to\oplus_{\mu\in\NN^t_n}\otimes \stensor{\mat{p_i}{C}}{\mu_i}
\end{equation}
Indeed, let $\{E_{(k,l)}\mid k,l=1,\ldots,m\}$ denote the standard basis of $\mathbb{C}^{m}$. This gives a basis $\{E_D \mid D\in \NN^{m\times m}_n\}$ given by $E_D:=\sum_{x\in ([m]\times [m])^n, w(x)=D} E_x$. Similarly, the standard basis $\{E_{(i,k,l}\mid i=1,\ldots,t, k,l=1,\ldots,p_i\}$ of $\oplus_{i=1}^t\C^{p_i\times p_i}$ gives a basis $\{E_D\mid D=D_1\oplus\cdots\oplus D_t\}$ of $\stensor{\oplus_{i=1}^t \C^{p_i\times p_i}}{n}$ where $E_D:=\sum_{x\in ([p_1]\times [p_1]\cup\cdots\cup [p_t]\times[p_t])^n\mid w(x)=D_1\oplus\cdots\oplus D_t}$. Then the isomorphism is given by
\begin{equation}
A_D\mapsto\oplus_\mu \delta_{\one^\transp D_1\one,\mu_1}\cdots\delta_{\one^\transp D_t\one,\mu_t}\otimes_{i=1}^t A_{D_i}.
\end{equation}
This yields an isomorphism 
\begin{equation}
\psi:\A\to \bigoplus_{\mu\in\NN^t_n}\bigotimes_{i=1}^t\stensor{\C^{p_i\times p_i}}{\mu_i}.
\end{equation}

To make this map explicit in terms of the given bases, we use the following lemma.
\begin{lemma}\label{basechange}
Let $V$ be a vector space with bases $\{u_1,\ldots,u_m\}$ and $\{v_1,\ldots,v_m\}$ that are related by $v_i=\sum_{j}c_{j,i}u_j$. Introducing indeterminates $x_1,\ldots,x_m$, the bases $(u_\mu)_{\mu\in \NN_n^m}$ and $(v_\nu)_{\nu\in\NN_n^m}$ of $\stensor{V}{n}$ are related by
\begin{equation}
v_\nu=\sum_{\mu}u_\mu\cdot (x_1c_{1,1}+\cdots+x_mc_{1,m})^{\mu_1}\cdots (x_1c_{m,1}+\cdots+x_mc_{m,m})^{\mu_t}[x^\nu].
\end{equation}
\end{lemma}
\begin{proof}
We compute $\sum_{\nu\in\NN_n^m}x^\nu v_\nu$.
\begin{eqnarray}
\sum_{\nu\in\NN_n^m}x^\nu v_\nu&=&\sum_{y\in \{1,\ldots,m\}^n} \otimes_{i=1}^n x_{y_i}v_{y_i}\nonumber\\
&=&\sum_{y\in \{1,\ldots,m\}^n} \otimes_{i=1}^n x_{y_i}\left(\sum_{j=1}^m c_{j,i}u_j\right)\nonumber\\
&=&\sum_{y,z\in \{1,\ldots,m\}^n} \otimes_{i=1}^n x_{y_i}c_{z_i,y_i}u_{z_i}\nonumber\\
&=&\sum_{z\in \{1,\ldots, m\}^n} \otimes_{i=1}^n u_{z_i}\left(\sum_{j=1}^m x_jc_{z_i,j}\right)\nonumber\\
&=&\sum_{\mu}u_\mu \left(\sum_{j=1}^m x_jc_{1,j}\right)^{\mu_1}\cdots\left(\sum_{j=1}^m x_jc_{m,j}\right)^{\mu_m}.
\end{eqnarray}
\end{proof}

Application of the lemma yields the following result.

\begin{theorem}\label{mappsi}
The map $\psi$ is given by
\begin{equation}
\psi(\sum_{\nu}y_{\nu}A_{\nu})=\bigoplus_{\mu} \sum_{D=(D_1,\ldots,D_t)} y_D \bigotimes_{i=1}^t A_{D_i},
\end{equation}
where $\one^{\transp}D_i\one=\mu_i$. The coefficients $y_D$ are given by 
\begin{equation}
y_D=\sum_{\nu}y_\nu \left(\prod_{j=1}^t\prod_{k,l=1}^{p_j} \left(\sum_{i=1}^s x_i(\phi_j(R_i))_{k,l}\right)^{(D_{j})_{k,l}}\right)[x^\nu].
\end{equation}
\end{theorem}
\begin{proof}
Expressing the basis $\{\psi(A_\nu)\mid \nu\in \NN^s_n\}$ in terms of the basis $\{A_D\}$ using Lemma \ref{basechange} yields the claimed result. 
\end{proof}

Using the block diagonalization for each of the algebras $\stensor{\mat{p_i}{\C}}{\mu_i}$ as described in section \ref{subsec:blockdiag}, we obtain a block diagonalization 
\begin{equation}
\A\to\bigoplus_{\mu}\bigoplus_{\lambda_1,\ldots,\lambda_t} \bigotimes_{i=1}^t \mat{\sst{\lambda_i}{p_i}}{\C}
\end{equation}
of $\A$.

\section{Examples from coding theory}

\subsection*{Terwilliger algebra of the binary Hamming scheme}
The algebra $\A:=\stensor{\mat{2}{\C}}{n}$ is referred to in the literature as the Terwilliger algebra of the Hamming scheme. The matrices $D\in P(n,2)$ are usually indexed by $3$ parameters $i,j,t$ by setting 
\begin{equation}
D_{i,j}^t:=\begin{pmatrix}n+t-i-j&i-t\\j-t&t\end{pmatrix}, \quad A_{i,j}^t:=A_{D_{i,j}^t}.
\end{equation}
Here $0\leq t\leq i,j$ and $i+j-t\leq n$. In \cite{Schrijver} an explicit block diagonalization was given and used to compute semidefinite programming bounds for binary codes. Here we show that our method gives the same block diagonalization.

Since $p=2$, the partitions $\lambda=(n-k,k)$ are indexed by $k=0,\ldots,\lfloor \frac{n}{2}\rfloor$. Give $k$, the semistandard tableaux of shape $(n-k,k)$ are indexed by $i=k,\ldots, n-k$ by placing a $2$ in $i-k$ of the last $n-2k$ boxes of the first row and in all $k$ boxes of the second row of the tableau. Lets denote this tableau by $t_{k,i}$ and denote $e_{t_{k,i}}$ by $e_{k,i}$. We have $D(t_{k,i},t_{k,k})=\bigl(\begin{smallmatrix}n-i\\i-k&k\end{smallmatrix}\bigr)$. Hence for give $i,j$:
\begin{equation}
\sum_D x^D\inner{A_De_{k,i}}{e_{k,j}}=\frac{2^k}{(i-k)!(j-k)!}d_{2\to 1}^{j-k}(d^*_{1\to 2})^{i-k}(x_{2,2}x_{1,1}-x_{1,2}x_{2,1})^k x_{1,1}^{n-2k}.
\end{equation}
It is easy to see that $d_{2\to 1}(x_{2,2}x_{1,1}-x_{2,1}x_{1,2})=d^{*}_{1\to 2}(x_{2,2}x_{1,1}-x_{2,1}x_{1,2})=0$. Hence we obtain:
\begin{eqnarray}
\sum_D x^D\inner{A_De_{k,i}}{e_{k,j}}&=&\frac{2^k}{(i-k)!(j-k)!}(x_{2,2}x_{1,1}-x_{1,2}x_{2,1})^k d_{2\to 1}^{j-k}(d^*_{1\to 2})^{i-k} x_{1,1}^{n-2k}\\
&=& \frac{2^k \tbinom{n-2k}{i-k}}{(j-k)!}(x_{2,2}x_{1,1}-x_{1,2}x_{2,1})^k d_{2\to 1}^{j-k} x_{1,2}^{i-k}x_{1,1}^{n-k-i}\nonumber\\
&=&2^k \tbinom{n-2k}{i-k} (x_{2,2}x_{1,1}-x_{1,2}x_{2,1})^k \sum_{s=0}^{j-k}\tbinom{n-k-i}{s}\tbinom{i-k}{j-k-s}.\nonumber
\end{eqnarray} 
In this sum, only monomials $x^D$ with $D=D_{i,j}^t$ for some $t$, occur and they have coefficient
\begin{equation}
\inner{A_De_{k,i}}{e_{k,j}}2^k \tbinom{n-2k}{i-k} \sum_{s=0}^{j-k}\tbinom{n-k-i}{s}\tbinom{i-k}{j-k-s}\tbinom{k}{j-t-s}(-1)^{j-t-s}.
\end{equation}
Note that for $i\neq j$ the vectors $e_{k,i}$ and $e_{k,j}$ are orthogonal as they have disjoint support. Taking $t=i=j$ we see that
\begin{equation}
\inner{e_{k,i}}{e_{k,i}}=2^k \tbinom{n-2k}{i-k}.
\end{equation}
It follows that the block diagonalization is given by:
\begin{equation}
\psi: A_{i,j}^t \mapsto \bigoplus_{k=0}^{\lfloor\frac{n}{2}\rfloor} \left(\delta_{i',i}\delta_{j',j}\tbinom{n-2k}{i-k}^{-1/2}\tbinom{n-2k}{j-k}^{-1/2}\beta_{i,j,k}^t\right)_{i',j'=k}^{n-k},
\end{equation}
where
\begin{equation}
\beta_{i,j,k}^t:=(-1)^{j-t-s}\tbinom{n-2k}{i-k} \sum_{s=0}^{j-k}\tbinom{n-k-i}{s}\tbinom{i-k}{j-k-s}\tbinom{k}{j-t-s}.
\end{equation}

\begin{remark}
This is the same block diagonalization as was given by Schrijver\cite{Schrijver}, except that there a different expression (but the same value!) was used for $\beta_{i,j,k}^t$, namely
\begin{equation}
\beta_{i,j,k}^t=\sum_{u=0}^n (-1)^{u-t}\tbinom{u}{t}\tbinom{n-2k}{u-k}\tbinom{n-k-u}{i-u}\tbinom{n-k-u}{j-u}.
\end{equation}
\end{remark}

\subsection*{More algebras for binary codes}
Fix a positive integer $t$ and let $X:=\{0,1\}^t$. The symmetric group of two elements, $S_2$, acts on $X$ by exchanging the symbols $0$ and $1$ (in all $t$ positions). This induces an action of $S_2$ on the set of $X\times X$ matrices by simultaneous permutation of the rows and columns. Let $\B:=\{A\in\C^{X\times X}\mid \text{$A$ is $S_2$-invariant}\}$ be the algebra of matrices invariant under this action. The algebra $\A_n:=\stensor{B}{n}$ can be indentified with the set of matrices with rows and columns indexed by all $t\times n$ binary matrices $I$ that are invariant under all permutations of the indices induced by either permuting the $n$ rows of $I$ or by action of $S_2$ on any subset of the columns of $I$.

Since $\B\cong \C^{\{0,1\}^{t-1}\times \{0,1\}^{t-1}}\oplus \C^{\{0,1\}^{t-1}\times \{0,1\}^{t-1}}$ by sending $\begin{pmatrix}A&B\\B&A\end{pmatrix}\to A+B\oplus A-B$, Theorem... gives an explicit block diagonalization of $\A_n$. The simplest case, $t=1$ gives a diagonalization of the Bose-Mesner algebra of the Hamming cube. In the case $t=2$ we see that $\A_n\cong \oplus_{i=0}^n \stensor{\C^{2\times 2},i}\otimes\stensor{\C^{2\times 2}}{n-i}$, a direct sum of tensor products of Terwilliger algebras for the Hamming scheme.
  
\subsection*{Terwilliger algebra of the nonbinary Hamming scheme}
Let $q\geq 3$ be an integer and let $\B\subset \mat{q}{\C}$ be the set of matrices with rows and columns indexed by $\{0,1,\ldots,q-1\}$ that are invariant under simultaneous permutation of the rows and columns in $\{1,\ldots, q-1\}$. The algebra $\B$ is easily seen to have dimension $5$, where $B_1,\ldots, B_5$ form a basis by defining
\begin{equation}
\sum_{i=1}^5 x_iB_i:=
\left(\begin{smallmatrix}
x_1&x_2&\cdots&\cdots&x_2\\
x_3&x_4&x_5&\cdots&x_5\\
\vdots&x_5&\ddots&\ddots&\vdots\\
\vdots&\vdots&\ddots&\ddots&x_5\\
x_3&x_5&\cdots&x_5&x_4
\end{smallmatrix}\right).
\end{equation}
A block diagonalization $\B\to \mat{2}{\C}\oplus\mat{1}{\C}$ is given by
\begin{equation}
\phi (x_1B_1+\cdots+x_5B_5)=\left(\begin{smallmatrix}x_1&x_2\sqrt{q-1}\\x_3\sqrt{q-1}&x_4+x_5(q-2)\end{smallmatrix}\right)\oplus\left(\begin{smallmatrix}x_4-x_5\end{smallmatrix}\right)
\end{equation}   
for $x_1,\ldots,x_5\in \C$. 

The algebra $\A_{q,n}:=\stensor{\B}{n}$ is known as the Terwilliger algebra of the nonbinary Hamming scheme. It can be used for deriving bounds on the size of nonbinary codes from semidefinite programming, see \cite{Gijswijt}.

Applying Theorem \ref{mappsi}, we obtain a block diagonalization of $\A_{q,n}$ given by
\begin{eqnarray}
\psi(A_\nu)&=&\bigoplus_{w=0}^n \sum_{D=\left(\begin{smallmatrix}a&b\\c&n-w-a-b-c\end{smallmatrix}\right)\oplus \left(\begin{smallmatrix}w\end{smallmatrix}\right)} y_D A_{\left(\begin{smallmatrix}a&b\\c&n-w-a-b-c\end{smallmatrix}\right)},\\
y_D&=&x_1^ax_2^bx_3^c(x_4+(q-2)x_5)^{n-w-a-b-c}(x_4-x_5)^{w}(q-1)^{\frac{b+c}{2}}[x^\nu].
\end{eqnarray} 
Restricting to $D$ with $y_D\neq 0$ we obtain:
\begin{equation}
\psi (A_\nu)=\bigoplus_{w=0}^n A_{\left(\begin{smallmatrix}\nu_1&\nu_2\\ \nu_3&\nu_4+\nu_5-w\end{smallmatrix}\right)}\sum_g \tbinom{\nu_4+\nu_5-w}{g}\tbinom{w}{\nu_5-g}(q-2)^g (-1)^{\nu_5-g}(q-1)^{\frac{\nu_2+\nu_3}{2}}.
\end{equation} 

Using the block diagonalization of $\stensor{\mat{2}{\C}}{n-w}$, we obtain a block diagonalization of $\A_{q,n}$. This agrees with the block diagonalization found in \cite{Gijswijt}.

\section*{Acknowledgments}
I want to thank Jan Draisma, Lex Schrijver and Frank Vallentin for helpful comments and discussions.

\end{document}